\documentclass[11pt,reqno]{amsart}
\usepackage{a4wide}
\usepackage{hyperref,amsmath,amsfonts,mathscinet,amssymb,amsthm,natbib}
\usepackage{accents}
\usepackage{enumitem}
\usepackage{constants}
\usepackage{tikz-cd}
\allowdisplaybreaks

\setlist[enumerate]{label={\rm(\roman*)}}

\hypersetup{colorlinks=true, linkcolor=red, filecolor=blue, urlcolor=cyan}

\theoremstyle{plain}
\newtheorem{theorem}{Theorem}[section]

\newtheorem*{theorem-A}{Theorem~A}
 
 \theoremstyle{definition}

 \newtheorem{remark}[theorem]{Remark}

\numberwithin{theorem}{section}
\numberwithin{equation}{section}

 \allowdisplaybreaks

\expandafter\let\expandafter\oldproof\csname\string\proof\endcsname
\let\oldendproof\endproof
\renewenvironment{proof}[1][\proofname]{%
  \oldproof[\bf #1]%
}{\oldendproof}

\def\urladdrname{{\texttt{ORCiD~}}}

\def\Z{\mathbb Z}

\def\esssup
{\operatornamewithlimits{ess\,\sup}}

\date{\today}

\begin{document}

\title{The Persson--Stepanov theorem revisited}

\author{Amiran Gogatishvili, Lubo\v s Pick, Hana Tur\v cinov\'a, and Tu\u{g}\c{c}e \"{U}nver}

\address{Amiran Gogatishvili, Institute of Mathematics of the Czech
 Academy of Sciences,
 \v Zitn\'a~25,
 115~67 Praha~1,
 Czech Republic,
 \urladdrname{0000-0003-3459-0355}}
\email{gogatish@math.cas.cz}

\address{Lubo\v s Pick, De\-pa\-rtment of Mathematical Analysis, Faculty of Mathematics and
Physics, Cha\-rles University, Sokolovsk\'a~83,
186~75 Praha~8, Czech Republic, 
\urladdrname{0000-0002-3584-1454} 
}
\email{pick@karlin.mff.cuni.cz}

\address{Hana Tur\v cinov\'a,
 Czech Technical University in Prague, Faculty of Electrical Engineering, Department of Mathematics, Technick\'a~2, 166~27 Praha~6, Czech Republic, \urladdrname{0000-0002-5424-9413}}
\email{hana.turcinova@fel.cvut.cz}

\address{Tu\u{g}\c{c}e \"{U}nver, Faculty of Engineering and Natural Sciences, Department of Mathematics,
Kiri\-kkale University, 71450, Yahsihan, Kirikkale, T\"{u}rkiye, \urladdrname{0000-0003-0414-8400} }
\email{tugceunver@kku.edu.tr}



\begin{abstract} 
    We develop a new proof of the result of L.-E.~Persson and V.D.~Stepanov \cite[Theorems 1 and 3]{Per:02}, which provides a characterization of a Hardy integral inequality involving two weights, and which can be applied to an effective treatment of the geometric mean operator. Our approach enables us to extend their result to the full range of parameters, in particular involving the critical case $p=1$, which was excluded in the original work. Our proof avoids all duality steps and discretization techniques and uses solely elementary means.
\end{abstract}

\thanks{This research was supported in part by grant no.~23-04720S of the Czech Science Foundation.
The research of  A.~Gogatishvili 
was partially supported by the Institute of Mathematics, CAS is supported by RVO:67985840, by  Shota Rustaveli National Science Foundation (SRNSF), grant no: FR22-17770.
 Part of this work was completed during A.~Gogatishvili's visit to Kirikkale University with the support of TUBITAK 2221 program Project No:  1059B212400044. }

\maketitle

\noindent\textbf{Dedication:
    The paper is a tribute to our dear friend Lars--Erik Persson at the occasion of his 80th birthday.}

\section{Introduction and the main results}

In this paper, we recall the result of L.-E.~Persson and V.D.~Stepanov stated in~\cite[Theorems 1 and 3]{Per:02}, in which a remarkable new characterization of the notorious two-weight Hardy inequality for functions defined on an interval was obtained. 

The Hardy inequality in question states that, given two parameters, $p$ and $q$, satisfying $1\le p, q<\infty$, and a pair of \emph{weights} (measurable functions which are positive and finite almost everywhere) $v$ and $w$ defined on an open interval $(a,b)$ (bounded or unbounded at each endpoint), then, under appropriate restrictions on $p,q,v,w$, there exists a positive constant $C$ such that the inequality
\begin{equation}\label{E:Hardy}
    \left(\int_a^b \left(\int_a^{t}f\right)^{q} w(t) \, dt\right)^{\frac{1}{q}} \leq C\left(\int_a^b f^{p}v\right)^{\frac{1}{p}}
\end{equation}
holds for all nonnegative measurable functions $f$. Here $C$ is allowed to depend on $p,q,v$ and $w$ but has to remain independent of $f$.

A characterization of the inequality~\eqref{E:Hardy}, together with its endless supply of modifications, generalizations and extensions, had been known, of course, long before, thanks to the efforts of many authors.  The history of pursuing inequality is over one century long and begins with classical 1920's results of G.H.~Hardy~\cite{Hardy:20}. In the 1950's, I.S.~Kac and M.G.~Kre\u{\i}n considered a special case of the weighted version~\cite{KacKre:58}. In the 1960's various observations were added by P.R.~Beesack~\cite{Bee:61}. Later, the so-called `convex case' ($p\le q$) was studied extensively~\cite{Tom:69,Tal:69,Muc:72,Bra:78,Kok:79,Rie:74} including various untitled and unpublished manuscripts such as those by M.~Artola and by D.W.~Boyd and J.A.~Erd\H{o}s. 

To summarize the situation in the convex case,~\eqref{E:Hardy} holds if and only if
\begin{equation*}
\sup_{t\in(a,b)}\left(\int_{t}^{b}w\right)^{\frac{1}{q}}
        \left(\int_{a}^{t}v^{1-p'}\right)^{\frac{1}{p'}}<\infty \quad\text{if $1<p\le q$}
\end{equation*}
and
\begin{equation*}
\sup_{t\in(a,b)}\left(\int_{t}^{b}w\right)^{\frac{1}{q}}
        \esssup_{s\in(a,t)}\frac{1}{v(s)}<\infty \quad\text{if $1=p\le q$.}
\end{equation*}
Here and throughout, if $p\in[1,\infty]$, then $p'$ denotes the conjugate exponent defined by $\frac{1}{p}+\frac{1}{p'}=1$. Observe that $1$ and $\infty$ are taken for conjugate exponents.

The `non-convex case' ($p>q$) was treated separately, and the results can be found, for example, in~\cite{Maz:11,Saw:84,Sin:91,SiSt:96,BeGr:06}. Here, for $1\le q<p<\infty$, the necessary and sufficient condition is, 
\begin{equation*}
\int_{a}^{b}\left(\int_{t}^{b}w\right)^{\frac{r}{q}}\left(\int_{a}^{t}v^{1-p'}\right)^{\frac{r}{q'}}v(t)^{1-p'}\,dt<\infty,
\end{equation*}
where $r=\frac{pq}{p-q}$, with appropriate modifications in the remaining cases.

The principal mission of L.-E.~Persson and V.D.~Stepanov in~\cite{Per:02} was to obtain a new characterization of a two-weight inequality for the~\emph{geometric mean operator}, $G$, defined at any admissible function $f$ on $(0,\infty)$ by
\begin{equation}\label{E:geometric}
    Gf(t) = \exp\left(\frac{1}{t}\int_{0}^{t}\log f(s)\,ds\right)
    \quad\text{for $t\in(0,\infty)$.}
\end{equation}
More precisely, they were searching for a condition that would characterize an inequality involving $G$ by means of constants having better stability properties than those known from the scheme developed earlier in~\cite{Pic:94}.
Investigation of weighted inequalities involving the operator~\eqref{E:geometric} also has a rich history. It began  with the  inequality of K.~Knopp~\cite{Kno:28}, see also T.~Carleman~\cite{Car:1923}, and over the following decades was studied by many~\cite{Hei:75,Coc:84,Lov:86,Hei:90,Opi:94,Pic:94,Jai:00,Jai:01}. The approach in~\cite{Per:02} was based on the combination of classical and modern ideas, namely the characterization of the best constant in a particular instance of~\eqref{E:Hardy} by G.A.~Bliss~\cite{Bli:30} and a related result by V.M.~Manakov~\cite{Man:92}. It turned out that a new necessary and sufficient condition of~\eqref{E:Hardy} was needed. The result can be summarized as follows (\cite{Per:02,Ste:94}).

\begin{theorem-A}[Persson and Stepanov]
    The inequality~\eqref{E:Hardy} holds 
    if and only if $A_{PS}<\infty$, where
    \begin{equation*}
        A_{PS} = \begin{cases}
            \sup\limits_{t \in (a,b)} \left(\int_{a}^{t}v^{1-p'}\right)^{-\frac{1}{p}} \left(\int_a^t \left(\int_{a}^{s}v^{1-p'}\right)^{q} w(s) \,ds\right)^{\frac{1}{q}}
                & \text{if $1<p\le q<\infty$,}
                    \\
\int_{a}^{b}\left(\int_{a}^{t}\left(\int_{a}^{s}v^{1-p'}\right)^qw(s)\,ds\right)^{\frac{r}{p}}w(t)\left(\int_{a}^{t}v^{1-p'}\right)^{q-\frac{r}{p}}\,dt
                & \text{if $0<q<p<\infty$, $p>1$.}
            \end{cases}
    \end{equation*}
\end{theorem-A}
We note that, throughout the text, by requiring that some expression, which defines certain balance condition (such as $A_{PS}$ above or $A_{\varepsilon}, B_{1,\varepsilon},B_{2,\varepsilon}$ below), is finite, we require at the same time that any its subexpression  is finite as well. 

In the course of the proof, duality methods were applied in the convex case, and a discretization technique was used in the non-convex one. The result avoids the case $p=1$, and, moreover, in the non-convex environment, the case when the quantity $\int_{a}^{b}v^{1-p'}$ is finite requires a different treatment from the case when it is not.

Our principal goal in this paper is to offer a lateral point of view to Theorem~A. First, we develop a completely new proof of the result. Our approach is based on certain ideas developed recently in~\cite{Gog:25}, and it has some pleasant features. Most importantly, it extends the result to the full range of parameters, including the limiting case $p=1$. Next, it establishes an entire scale of characterizing conditions, a fact that gives one certain versatility when applying the result. Furthermore, both the statement and the proof are uniform with respect to all parameter values. Hence, the case $p=1$ no longer has to be singled out as a special case. Our proof unifies the argument that had to be split in~\cite[Theorem~3]{Per:02} into two separate cases, so the distinction is no longer necessary. Last but not least, it completely avoids all kinds of duality or discretization techniques and rests only upon elementary tools such as the Minkowski integral inequality and trivial monotonicity properties of functions given by integrals.

We shall use the following notation. Throughout the paper, we assume that $-\infty\le a<b\le\infty$. For a weight $v$ on $(a,b)$ and $t\in(a,b)$, we set
\begin{equation}\label{E:V}
    V(t) =
        \begin{cases}
            \left(
            \int_{a}^{t}v^{1-p'} \right)^{\frac{1}{p'}} 
            &\text{if $p\in(1,\infty)$,}
                \\            \operatornamewithlimits{ess\,sup}\limits_{s\in(a,t)}\frac{1}{v(s)} &\text{if $p=1$.}
        \end{cases}
\end{equation}
We denote by $\mathcal{M}^+(a, b)$ the collection of all nonnegative measurable functions on $(a,b)$. We use the symbol $\lesssim$ when the expression on the left of it is majorized by constant times that on the right. We write $\approx$ to represent the conjunction of $\lesssim$ and $\gtrsim$. We will denote by $LHS_{(*)}$ the expression standing on the left-hand side of inequality $(*)$, and analogously for $RHS_{(*)}$.

We will formulate the main result as two theorems, broken down by comparing the parameters $p$ and $q$. We begin with the convex case. 

\begin{theorem}\label{T:1}
Let $1\le p\leq q< \infty$ and let $v, w$ be weights on $(a,b)$.
Then the following statements are equivalent:

{\rm(i)} There exists a positive constant $C$ such that the inequality~\eqref{E:Hardy}
holds for all $f \in \mathcal{M}^+(a, b)$.

{\rm(ii)} For every $\varepsilon\in(0,\infty)$ one has $A_{\varepsilon} <\infty$, where
\begin{equation*}
A_{\varepsilon} = \sup\limits_{t \in (a, b)} V(t)^{-\varepsilon} \left(\int_a^t V^{q(\varepsilon+1)} w\right)^{\frac{1}{q}}.
\end{equation*}

{\rm(iii)} There exists $\varepsilon\in(0,\infty)$ such that $A_{\varepsilon} <\infty$.

Moreover, for every fixed $\varepsilon$, the least constant $C$ in \eqref{E:Hardy} satisfies $C\approx A_{\varepsilon}$ with constants of equivalence depending only on $v,w,p,q$ and $\varepsilon$.
\end{theorem}

\begin{remark}
    For $1<p\le q<\infty$, one has $A_{p'-1}=A_{PS}$, hence Theorem~\ref{T:1} recovers the convex part of Theorem~A.
\end{remark}

As usual, things are more complicated in the non-convex case contained in the next theorem.

\begin{theorem} \label{T:2}
Let $1\le p<\infty$, $0 < q < p < \infty$, $r=\frac{pq}{p-q}$ and $v, w$ be weights on $(a,b)$.
Then the following statements are equivalent:

{\rm(i)} There exists a positive constant $C$ such that the inequality \eqref{E:Hardy}
holds for all $f \in \mathcal{M}^+(a, b)$.

{\rm(ii)}  For every $\varepsilon\in(0,\infty)$ one has $B_{1,\varepsilon} <\infty$,  where
\begin{equation*}
B_{1,\varepsilon} = \left(\int_a^b \bigg(\int_a^t V^{q(\varepsilon+1)} w\bigg)^{\frac{r}{p}} V(t)^{q(\varepsilon+1)-\varepsilon r} w(t) \, dt\right)^{\frac{1}{r}}.
\end{equation*}

{\rm(iii)}  For every $\varepsilon\in(0,\infty)$ one has $B_{2,\varepsilon} <\infty$, where
\begin{align*}
B_{2,\varepsilon} &= \left(\int_a^b \bigg(\int_a^t V^{q(\varepsilon +1)} w\bigg)^{\frac{r}{q}} \, d \left[-V(t)^{-\varepsilon r}\right]  \right)^{\frac{1}{r}} + V(b)^{-\varepsilon} \bigg(\int_a^b V^{q(\varepsilon+1)} w\bigg)^{\frac{1}{q}}.
\end{align*}

{\rm(iv)} There exists $\varepsilon\in(0,\infty)$ such that $B_{1,\varepsilon} <\infty$.

{\rm(v)} There exists $\varepsilon\in(0,\infty)$ such that $B_{2,\varepsilon} <\infty$.

Moreover, for every fixed $\varepsilon>0$, the least constant $C$ in \eqref{E:Hardy} satisfies $C\approx B_{1,\varepsilon}\approx B_{2,\varepsilon}$ with equivalence constants dependent only on $p,q,v,w$ and $\varepsilon$.
\end{theorem}

\begin{remark}
    For $0<q<p<\infty$ and $p>1$, one has $B_{1,p'-1}=A_{PS}$, hence Theorem~\ref{T:2} recovers the non-convex part of Theorem~A.
\end{remark}

\section{Proofs}

\begin{proof}[Proof of Theorem~\ref{T:1}]
Clearly, $(ii)$ implies $(iii)$.

$(iii)\Rightarrow (i)$: Let $\varepsilon \in (0,\infty)$ be such that $A_{\varepsilon}<\infty$, and let $0< \lambda <\min(1,\varepsilon)$. We claim that
\begin{equation}\label{E:1}
    \int_a^t f  \lesssim \bigg(\int_a^t f^p V^{\lambda p } v\bigg)^{\frac{1}{p}} V(t)^{1-\lambda}
    \quad\text{for every $t\in(a,b)$.}
\end{equation}
The proof of~\eqref{E:1} (at least with $\varepsilon\in(0,1)$, $a=0$ and $b=\infty$) can be found in~\cite[(2.1)]{Gog:25}, but since we need a certain nontrivial modification of it, we will give a detailed proof here. 
Fix $t\in(a,b)$. If $p\in(1,\infty)$, then, by H\"older's inequality,
\begin{align*}
    \int_{a}^{t}f & = \int_{a}^{t}fV^{\lambda}v^{\frac{1}{p}}V^{-\lambda }v^{-\frac{1}{p}}
        \le \left(\int_{a}^{t}f^{p}V^{\lambda p}v\right)^{\frac{1}{p}}
        \left(\int_{a}^{t}V^{-\lambda  p'}v^{1-p'}\right)^{\frac{1}{p'}}.
\end{align*}
By a change of variables, we obtain
\begin{equation*}
    \int_{a}^{t}V^{-\lambda p'}v^{1-p'} = \int_{a}^{t} \left(\int_{a}^{s}v^{1-p'}\right)^{-\lambda }v(s)^{1-p'}\,ds = \frac{1}{1-\lambda }\left(\int_{a}^{t}v^{1-p'}\right)^{1-\lambda}
    = \frac{1}{1-\lambda } V(t)^{(1-\lambda )p'},
\end{equation*}
hence
\begin{equation*}
    \int_{a}^{t}f \lesssim \left(\int_{a}^{t}f^{p}V^{\lambda  p}v\right)^{\frac{1}{p}}V(t)^{1-\lambda },
\end{equation*}
and~\eqref{E:1} follows. If $p=1$, then
\begin{equation*}
    \int_{a}^{t}f = \int_{a}^{t}fV^{\lambda}vv^{-1}V^{-\lambda}.
\end{equation*}
Since, for every $s\in(a,t)$, one has, owing to the definition of $V$ and monotonicity,
\begin{equation*}
    v(s)^{-1}V(s)^{-\lambda}\le V(s)^{1-\lambda}\le V(t)^{1-\lambda},
\end{equation*}
we, in fact, get
\begin{equation*}
    \int_{a}^{t}f 
        \le \left(\int_{a}^{t}fV^{\lambda }v\right) V(t)^{1-\lambda},
\end{equation*}
which is~\eqref{E:1} once again. 

Using~\eqref{E:1}, now proved, we have
\begin{align}
{LHS}_{\eqref{E:Hardy}}^q & \lesssim \int_a^b \bigg(\int_a^t f^p V^{\lambda p } v\bigg)^{\frac{q}{p}} V(t)^{q(1-\lambda)} w(t) \, dt \notag \\
& = \int_a^b \bigg(\int_a^t f^p V^{\lambda p } v\bigg)^{\frac{q}{p}} V(t)^{-q(\lambda+\varepsilon)} V(t)^{q(\varepsilon+1)} w(t) \, dt. \label{E:LHS}
\end{align}
For every $t\in (a,b)$, one has
\begin{equation*}
V(t)^{-q(\lambda+\varepsilon)} \approx \left(\int_t^b d\bigg[-V(s)^{-p(\lambda+\varepsilon)} \bigg]\right)^{\frac{q}{p}} + V(b)^{-q(\lambda+\varepsilon)}.
\end{equation*}
Therefore,
\begin{align}\label{E:II}
LHS_{\eqref{E:Hardy}}^q & \lesssim  \int_a^b \bigg(\int_a^t f^p V^{\lambda p} v\bigg)^{\frac{q}{p}} \left(\int_t^b d\bigg[-V(s)^{-p(\lambda+\varepsilon)} \bigg]\right)^{\frac{q}{p}} V(t)^{q(\varepsilon+1)} w(t) \, dt \notag\\
& \quad +    \int_a^b \bigg(\int_a^t f^p V^{\lambda p } v\bigg)^{\frac{q}{p}} V(b)^{-q(\lambda+\varepsilon)} V(t)^{q(\varepsilon+1)} w(t) \, dt =: I + II.
\end{align}
Using monotonicity, we have
\begin{align*}
I \leq \int_a^b  \left( \int_t^b  \bigg(\int_a^s f^p V^{\lambda p} v\bigg)d\bigg[-V(s)^{-p(\lambda+\varepsilon)}\bigg]\right)^{\frac{q}{p}} V(t)^{q(\varepsilon+1)} w(t) \, dt.
\end{align*}
Since $p \leq q$, applying Minkowski's integral inequality with $\frac qp \geq 1$, we obtain
\begin{align*}
I \leq \left(\int_a^b \bigg(\int_a^s f^p V^{\lambda p} v\bigg) \left( \int_a^s V(t)^{q(\varepsilon+1)} w(t) \, dt \right)^{\frac{p}{q}} d\bigg[-V(s)^{-p(\lambda+\varepsilon)} \bigg] \right)^{\frac{q}{p}}.
\end{align*}
Owing to the fact that $A_\varepsilon<\infty$, for each $s \in (a,b)$ we have
\begin{equation}
\left(\int_a^s V(t)^{q(\varepsilon+1)} w(t) dt\right)^{\frac{p}{q}} \leq A_\varepsilon^p V(s)^{\varepsilon p}.
\end{equation}
Therefore,  Fubini's theorem gives 
\begin{align*}
I &\leq A_\varepsilon^q \left(\int_a^b \bigg(\int_a^s f^p V^{\lambda p} v\bigg) V(s)^{\varepsilon p} \,  d\bigg[-V(s)^{-p(\lambda+\varepsilon)} \bigg] \right)^{\frac{q}{p}} 
    \\
&=A_\varepsilon^q \left(\int_a^b  f(t)^p V(t)^{\lambda p} v(t) \int_t^b V(s)^{\varepsilon p}  d\bigg[-V(s)^{-p(\lambda+\varepsilon)}\bigg ] dt \right)^{\frac{q}{p}}
    \\
& \approx A_\varepsilon^q \left(\int_a^b f(t)^p  V(t)^{\lambda p} v(t) \bigg[ V(t)^{-\lambda p} -V(b)^{-\lambda p} \bigg] dt \right)^{\frac{q}{p}} 
    \\
&\leq A_\varepsilon^q \left(\int_a^b f(t)^p  v(t) dt \right)^{\frac{q}{p}}.
\end{align*}
If $V(b)=\infty$, then $II=0$, and there is nothing to prove. Assume thus that $V(b)<\infty$. Then, by monotonicity, we have, 
\begin{align*}
II & = V(b)^{-q(\lambda+\varepsilon)} 
\int_a^b \bigg(\int_a^t f^p V^{\lambda p } v\bigg)^{\frac{q}{p}}  V(t)^{q(\varepsilon+1)} w(t) \, dt   \\
& \leq  V(b)^{-q\varepsilon} 
\int_a^b V(t)^{q(\varepsilon+1)} w(t) \, dt
\bigg(\int_a^b f^p v\bigg)^{\frac{q}{p}} 
    \le A_\varepsilon^q \left(\int_a^b f(t)^p  v(t) dt \right)^{\frac{q}{p}}.
\end{align*}
Then, in view of \eqref{E:II}, we arrive at
\begin{equation*}
    LHS_{\eqref{E:Hardy}}^q \lesssim I + II \lesssim A_\varepsilon^q \left(\int_a^b f(t)^p  v(t) dt \right)^{\frac{q}{p}}.
\end{equation*}
In other words, inequality \eqref{E:Hardy} holds and the best constant in \eqref{E:Hardy} satisfies $C\lesssim A_\varepsilon$.

$(i)\Rightarrow (ii)$: Fix $\varepsilon\in(0,\infty)$ and assume that \eqref{E:Hardy} holds for all $f \in \mathcal{M}^+(a, b)$. Fix $t\in(a,b)$. If $p\in (1,\infty)$, then, testing the inequality with the function $f_t =V^{1+\varepsilon-p'} v^{1-p'} \chi_{(a,t)}$, we obtain
\begin{align}
LHS_{\eqref{E:Hardy}} \geq \left(\int_a^t \left(\int_a^{s}f_t\right)^{q} w(s) \, ds\right)^{\frac{1}{q}} \approx \left(\int_a^t V(s)^{(\varepsilon+1)q} w(s) \, ds\right)^{\frac{1}{q}},
\end{align}
whereas
\begin{align}
RHS_{\eqref{E:Hardy}}
\approx C V(t)^{\varepsilon},
\end{align}
and the constants of equivalence in each of the latter two estimates are independent of $t$.
Consequently, (i) implies
\begin{equation*}
    \left(\int_a^t V^{(\varepsilon+1)q} w\right)^{\frac{1}{q}} \lesssim C V(t)^{\varepsilon} \quad \text{for every $t\in (a,b)$,}
\end{equation*}
hence (ii) holds with $A_\varepsilon\lesssim C$.

Let $p=1$. Fix some $\sigma>1$ and define
\begin{equation*}
    E_k = \{t\in(a,b) : \sigma^k<V(t)\le \sigma^{k+1}\} \quad \text{for $k\in\Z$.}
\end{equation*}
Set $\mathbb A = \{k\in\Z : |E_k|>0\}$. 
Then $(a,b)=\bigcup_{k\in\mathbb A}E_k$, in which the equality is understood up to a set of measure zero, in which the union is disjoint and each $E_k$ is a nondegenerate interval (which could be either open or closed at each end) with endpoints $a_k$ and $b_k$, $a_k<b_k$. For every $k\in\mathbb A$, we find $\delta_k > 0$ so that $a_k+\delta_k<b_k$, 
and then we define the set
\begin{equation*}
    G_k = \left\{t\in(a_k,a_k+\delta_k) : \frac{1}{v(t)} >\sigma^{k}\right\}.
\end{equation*}
Since $V$ is non-decreasing and left-continuous, $|G_k|>0$ for every $k\in\mathbb A$. Fix $t\in(a,b)$ and set \begin{equation*}
    f_t(s) = \chi_{(a,t)}(s)\sum_{k\in\mathbb A} \frac{V(s)^{\varepsilon}\chi_{_{G_k}}(s)}{v(s)|G_k|}
    \quad\text{for $s\in(a,b)$.} 
\end{equation*}
There is a uniquely defined $k\in\mathbb A$ such that $t\in(a_k,b_k]$. Consequently
\begin{align*}
      \int_{a}^{b}f_tv  & \le \sum_{j\in\mathbb A,\ j\le k}\frac{1}{|G_j|}\int_{G_j}V^{\varepsilon }
       \le \sum_{j=-\infty}^{k}\sigma^{\varepsilon (j+1)} = \frac{\sigma^{\varepsilon (k+2)}}{\sigma^{\varepsilon }-1}\lesssim V(t)^{\varepsilon }.
\end{align*}
On the other hand, for every $s\in(a_k,b_k]$
\begin{align*}
  \int_{a}^{s}f_t\ge \frac{1}{|G_k|}\int_{G_k}V(y)^{\varepsilon }\frac{1}{v(y)}\, dy \ge \sigma^{k(\varepsilon +1)}=\frac{\sigma^{(k+1)(\varepsilon +1)}}{ \sigma^{\varepsilon +1}}\ge \frac{1}{ \sigma^{\varepsilon +1}}V(s)^{\varepsilon +1}.
\end{align*}
The last two estimates, together with the validity of the inequality \eqref{E:Hardy}, yield
\begin{equation}\label{E:upper-estimate-of-integral-1}
     \bigg( \int_{a}^{t}V^{(\varepsilon +1)q}w\bigg)^{\frac{1}{q}} \lesssim C V(t)^{\varepsilon } \quad\text{for $t\in(a,b)$,}
\end{equation}
and (ii) follows with $A_\varepsilon \lesssim C$.
\end{proof}

\begin{proof}[Proof of Theorem~\ref{T:2}] 
Clearly, $(ii)$ implies $(iv)$ and  $(iii)$ implies $(v)$.

$(ii) \Leftrightarrow(iii)$ and $(iv) \Leftrightarrow(v)$: Fix $\varepsilon\in(0,\infty)$. 
If $B_{1,\varepsilon}<\infty$, then we have 
\begin{equation*}
    \lim_{s\to a}
    V(s)^{-\varepsilon}
    \left(\int_a^s V(t)^{q(\varepsilon+1)} w(t) \, dt\right)^{\frac{1}{q}}=0,
\end{equation*}

which follows from the following simple estimate: for a fixed $x\in(a,b)$, we have 
\begin{align*}
   &\left(\int_a^x \bigg(\int_a^t V^{q(\varepsilon+1)} w\bigg)^{\frac{r}{p}} V(t)^{q(\varepsilon+1) -\varepsilon r} w(t) \, dt\right)^{\frac{1}{r}} 
    \\
    &\qquad \ge
    \sup_{s\in(a,x)}
    \left(\int_a^s \bigg(\int_a^t V^{q(\varepsilon+1)} w\bigg)^{\frac{r}{p}} V(t)^{q(\varepsilon+1)}  V(s)^{-\varepsilon r}w(t) \, dt\right)^{\frac{1}{r}} 
    \\
    &\qquad\approx
    \sup_{s\in(a,x)}
    V(s)^{-\varepsilon}
    \left(\int_a^s V(t)^{q(\varepsilon+1)} w(t) \, dt\right)^{\frac{1}{q}}.
\end{align*}
Thus, by integrating by parts, we get 
\begin{align}
&\int_a^b \bigg(\int_a^t V^{q(\varepsilon+1)} w\bigg)^{\frac{r}{p}} V(t)^{q(\varepsilon+1) -\varepsilon r} w(t) \, dt
=\frac{q}{r}\int_a^b  V(t)^{-\varepsilon r}  \, d\bigg[\int_a^t V^{q(\varepsilon+1)} w\bigg]^{\frac{r}{q}} \notag
    \\
    &=\frac{q}{r}
\int_a^b \bigg(\int_a^t V^{q(\varepsilon+1)}  w\bigg)^{\frac{r}{q}} \, d\left[ -V(t)^{-\varepsilon r }\right] 
    +\frac{q}{r} V(b)^{-\varepsilon r} \bigg(\int_a^b V^{q(\varepsilon +1)} w\bigg)^{\frac{r}{q}},
   \label{B1B2}\end{align}
hence $B_{2,\varepsilon}\approx B_{1,\varepsilon}<\infty$.
This shows the implications $(ii)\Rightarrow(iii)$ and $(iv) \Rightarrow(v)$. Conversely, if $B_{2,\varepsilon}<\infty$, we have, for every $x\in(a,b)$,
\begin{align*}
    B_{2,\varepsilon}^r &\approx \int_a^b \bigg(\int_a^t V^{q(\varepsilon+1)}  w\bigg)^{\frac{r}{q}} \, d\left[ -V(t)^{-\varepsilon r }\right] + V(b)^{-\varepsilon r} \bigg(\int_a^b V^{q(\varepsilon +1)} w\bigg)^{\frac{r}{q}}\\
    &\ge  \int_x^b \bigg(\int_a^t V^{q(\varepsilon+1)}  w\bigg)^{\frac{r}{q}} \, d\left[ -V(t)^{-\varepsilon r }\right] + V(b)^{-\varepsilon r} \bigg(\int_a^x V^{q(\varepsilon +1)} w\bigg)^{\frac{r}{q}}\\
    &\ge  \bigg(\int_a^x V^{q(\varepsilon+1)}  w\bigg)^{\frac{r}{q}} V(x)^{-\varepsilon r}.
\end{align*}
By integration by parts, once again,
\begin{align*}
\left(\int_a^b \bigg(\int_a^t V^{q(\varepsilon+1)} w\bigg)^{\frac{r}{p}} V(t)^{q(\varepsilon+1) -\varepsilon r} w(t) \, dt\right)^{\frac{1}{r}} 
&=\left(\frac{q}{r}\int_a^b  V(t)^{-\varepsilon r}  \, d\bigg[\int_a^t V^{q(\varepsilon+1)} w\bigg]^{\frac{r}{q}}\right)^{\frac{1}{r}}\\
&\hskip-5cm \lesssim\left(\int_a^b \bigg(\int_a^t V^{q(\varepsilon+1)}  w\bigg)^{\frac{r}{q}} \, d\left[ -V(t)^{-\varepsilon r }\right]\right)^{\frac{1}{r}} 
 + V(b)^{-\varepsilon} \bigg(\int_a^b V^{q(\varepsilon +1)} w\bigg)^{\frac{1}{q}},
\end{align*}
hence $B_{1,\varepsilon}\lesssim B_{2,\varepsilon}<\infty$.
This shows the converse implications $(iii)\Rightarrow(ii)$ and $(v) \Rightarrow(iv)$. Moreover,~\eqref{B1B2} shows that $B_{1,\varepsilon}\approx B_{2,\varepsilon}$.

$(v) \Rightarrow (i)$: Let $\varepsilon\in(0,\infty)$ be such that $B_{2,\varepsilon} < \infty$. Fix $\lambda\in(0,\varepsilon)$.
Observe that, for $t\in (a,b)$, one has
\begin{equation}\label{E:v-i}
V(t)^{-q(\lambda+\varepsilon)} \approx \int_t^b V(s)^{-q(\lambda +\varepsilon)+\varepsilon r} \, d\left[-V(s)^{-\varepsilon r} \right] + V(b)^{-q(\lambda+\varepsilon )}.
\end{equation}
Note that~\eqref{E:LHS} is still true. Thus, combining~\eqref{E:LHS} with~\eqref{E:v-i}, we get
\begin{align*}
LHS_{\eqref{E:Hardy}}^q & \lesssim  \int_a^b \bigg(\int_a^t f^p V^{\lambda p} v\bigg)^{\frac{q}{p}} \left(\int_t^b V(s)^{-q(\lambda+\varepsilon) +\varepsilon r} \, d\left[-V(s)^{-\varepsilon r} \right]\right) V(t)^{q(\varepsilon+1)} w(t) \, dt \notag\\
& \quad +    \int_a^b \bigg(\int_a^t f^p V^{\lambda p} v\bigg)^{\frac{q}{p}} V(b)^{-q(\lambda +\varepsilon)} V(t)^{q(\varepsilon+1)} w(t) \, dt =: III + II,
\end{align*}
where $II$ is as in \eqref{E:II}.
Let us start with estimating $III$. Using monotonicity and then Fubini's theorem, we have
\begin{align*}
III & \leq  \int_a^b  \left(\int_t^b \bigg(\int_a^s f^p V^{\lambda p} v\bigg)^{\frac{q}{p}} V(s)^{-q(\lambda +\varepsilon)+\varepsilon r} \, d\left[-V(s)^{-\varepsilon r} \right] \right)V(t)^{q(\varepsilon +1)} w(t) \, dt \\
& = \int_a^b  \bigg(\int_a^s f^p V^{\lambda p} v\bigg)^{\frac{q}{p}} V(s)^{-q(\lambda +\varepsilon)+\varepsilon r} \left(\int_a^s V^{q(\varepsilon+1)} w\right) \, d\left[-V(s)^{-\varepsilon r} \right].
\end{align*}
Now, by H\"{o}lder's inequality with the exponents $(\frac pq, \frac rq)$, Fubini's theorem and the definition of $B_{2,\varepsilon}$, we have
\begin{align*}
III&\leq  \left( \int_a^b \left(\int_a^s f^p V^{\lambda p} v\right) V(s)^{\frac{p}{q}(-q(\lambda +\varepsilon)+\varepsilon r)}  \, d\left[-V(s)^{-\varepsilon r} \right] \right)^{\frac{q}{p}} \\
& \qquad\times \left(\int_a^b \left(\int_a^s V^{q(\varepsilon+1)} w\right)^{\frac{r}{q}} \, d\left[-V(s)^{-\varepsilon r} \right] \right)^{\frac{q}{r}}
    \\
&\leq B_{2,\varepsilon}^q \left( \int_a^b f(t)^p V(t)^{\lambda p} v(t)
\left(\int_t^b  V(s)^{\frac{p}{q}(-q(\lambda +\varepsilon) +\varepsilon r)}  \, d\left[-V(s)^{-\varepsilon r} \right]  \right) \, dt  \right)^{\frac{q}{p}} 
    \\
&= B_{2,\varepsilon}^q \left( \int_a^b f(t)^p V(t)^{\lambda p} v(t)
\left(\int_t^b  V(s)^{-p\lambda +\varepsilon r}  \, d\left[-V(s)^{-\varepsilon r} \right]  \right) \, dt  \right)^{\frac{q}{p}} \\
& \lesssim B_{2,\varepsilon}^q \left(\int_a^b f(t)^p  v(t) dt \right)^{\frac{q}{p}}
.
\end{align*}
Next, if $V(b)=\infty$, then $II=0$. If $V(b)<\infty$, then
\begin{align*}
II & \le V(b)^{-q\varepsilon} \left(\int_a^b V^{q(\varepsilon+1)} w\right)
\left(\int_{a}^{b}f^p v\right)^{\frac{q}{p}} 
 \leq B_{2,\varepsilon}^q \left(\int_a^b f^{p}v\right)^{\frac{q}{p}}.
\end{align*}
Therefore, we have
\begin{equation*}
LHS_{\eqref{E:Hardy}}^q \lesssim III + II \lesssim B_{2,\varepsilon}^q \left(\int_a^b f^{p}v\right)^{\frac{q}{p}},
\end{equation*}
whence the inequality \eqref{E:Hardy} holds, and the best constant in \eqref{E:Hardy} satisfies $C\lesssim  B_{2,\varepsilon}
$.

$(i) \Rightarrow (ii)$: 
Assume now that the inequality \eqref{E:Hardy} holds for all $f \in \mathcal{M}^+(a,b)$. Let first $p\in(1,\infty)$. Let $v_1$, $w_1$ be weights satisfying $w_1 \leq w$, $v \leq v_1$, $\int_a^b w_1 <\infty$ and $\int_{a}^{b}v_1^{1-p'}<\infty$. Set
\begin{equation}\label{E:V1-p-large}
V_1(t) :=\left( \int_{a}^{t}v_1^{1-p'} \right)^{\frac{1}{p'}} \quad\text{for $t\in[a,b]$.}
\end{equation}
Then
\begin{equation}\label{E:hardy-new}
\left(\int_a^b \left(\int_a^{t}f\right)^{q} w_1(t) \, dt\right)^{\frac{1}{q}} \leq C\left(\int_a^b f^{p}v_1\right)^{\frac{1}{p}}
\end{equation}
holds for all $f \in \mathcal{M}^+(a, b)$. 
Denote by
\begin{equation}\label{B-1-tilde}
\widetilde{B}_{1,\varepsilon} = \left(\int_a^b \bigg(\int_a^t V_1^{q(\varepsilon+1)} w_1\bigg)^{\frac{r}{p}} V_1(t)^{q(\varepsilon+1)-\varepsilon r} w_1(t) \, dt\right)^{\frac{1}{r}}
\end{equation}
and
\begin{align}\label{B-2-tilde}
\widetilde{B}_{2,\varepsilon} &= \left(\int_a^b \bigg(\int_a^t V_1^{q(\varepsilon +1)} w_1\bigg)^{\frac{r}{q}} \, d \left[-V_1(t)^{-\varepsilon r}\right]  \right)^{\frac{1}{r}} + V_1(b)^{-\varepsilon} \bigg(\int_a^b V_1^{q(\varepsilon+1)} w_1\bigg)^{\frac{1}{q}}.
\end{align}
Observe that, by monotonicity,
\begin{align}
\widetilde{B}_{2,\varepsilon} &\leq \bigg(\int_a^b  w_1\bigg)^{\frac{1}{q}} \left(\int_a^b  V_1(t)^{r(\varepsilon +1)} \, d \left[-V_1(t)^{-\varepsilon r}\right]  \right)^{\frac{1}{r}} + V_1(b)\bigg(\int_a^b w_1\bigg)^{\frac{1}{q}}\notag\\
 &\approx V_1(b)\bigg(\int_a^b w_1\bigg)^{\frac{1}{q}}\notag\\
 & <\infty.\label{est_B2}
\end{align}
Setting $f(x)= v_1(x)^{1-p'}$ for $x\in (a,b)$ and plugging this $f$ into~\eqref{E:hardy-new}, one obtains
\begin{equation}\label{E:new-1}
\left(\int_a^b V_1(x)^{qp'} w_1(x) \, dx\right)^{\frac{1}{q}} \leq C V_1(b)^{\frac{p'}{p}}.
\end{equation}
On the other hand, testing inequality \eqref{E:hardy-new} with
\begin{equation*}
f(x) = \left(\int_x^b \left(\int_a^t V_1^{q(\varepsilon+1)} w_1\right)^{\frac{r}{q}} V_1(t)^{-p'} d\left[-V_1(t)^{-\varepsilon r}\right] \right)^{\frac{1}{p}} v_1(x)^{1-p'}, \quad x\in (a,b),
\end{equation*}
we have, by Fubini's theorem,
\begin{align*} 
\left(\int_a^b f^p v_1\right)^{\frac{1}{p}} = \left(\int_a^b \left(\int_a^t V_1^{q(\varepsilon+1)} w_1\right)^{\frac{r}{q}} d\left[-V_1(t)^{-\varepsilon r}\right] \right)^\frac{1}{p}
\le \widetilde{B}_{2,\varepsilon}^{\frac{r}{p}}.
\end{align*}
Moreover, using monotonicity, we obtain
\begin{align*}
LHS_{\eqref{E:hardy-new}} 
& \geq 
\left( \int_a^b \left(\int_x^b 
\left(\int_a^t V_1^{q(\varepsilon+1)} w_1\right)^{\frac{r}{q}} V_1(t)^{-p'} d\left[-V_1(t)^{-\varepsilon r}\right] \right)^{\frac{q}{p}} V_1(x)^{q p'} w_1(x) dx \right)^{\frac{1}{q}} \nonumber\\
& \geq 
\left( \int_a^b  \left(\int_a^x V_1^{q(\varepsilon+1)} w_1\right)^{\frac{r}{p}} 
\left(\int_x^b  V_1(t)^{-p'}  d\left[-V_1(t)^{-\varepsilon r}\right] \right)^{\frac{q}{p}} V_1(x)^{qp'} w_1(x) dx \right)^{\frac{1}{q}} \nonumber\\
& \approx  \left( \int_a^b  
\left(\int_a^x V_1^{q(\varepsilon+1)} w_1\right)^{\frac{r}{p}} 
\left(V_1(x)^{-p'-\varepsilon r}-V_1(b)^{-p'-\varepsilon r} \right)^{\frac{q}{p}} V_1(x)^{qp'} w_1(x)  \,dx \right)^{\frac{1}{q}}.
\label{E:i-v}
\end{align*}
Thus, owing to~\eqref{E:hardy-new}, we get
\begin{equation}\label{E:new-2}
\left( \int_a^b  
\left(\int_a^x V_1^{q(\varepsilon+1)} w_1\right)^{\frac{r}{p}} 
\left(V_1(x)^{-p'-\varepsilon r}-V_1(b)^{-p'-\varepsilon r} \right)^{\frac{q}{p}} V_1(x)^{qp'} w_1(x)  \,dx \right)^{\frac{1}{q}} \lesssim C \widetilde{B}_{2,\varepsilon}^{\frac{r}{p}}.
\end{equation}
Note that $\widetilde{B}_{2,\varepsilon}$ is finite (which in fact was the principal reason for introducing $v_1$ and $w_1$ in the first place). Therefore, in view of \eqref{E:new-2} and \eqref{E:new-1}, one can write
\begin{align*}
\widetilde{B}_{1,\varepsilon}^{\frac{r}{q}}
&=\left( \int_a^b  \left(\int_a^x V_1^{q(\varepsilon+1)} w_1 \right)^{\frac{r}{p}} V_1(x)^{-\varepsilon r} V_1(x)^{q(\varepsilon+1)} w_1(x) dx \right)^{\frac{1}{q}}
    \\
&\approx 
\left( \int_a^b \left(\int_a^x V_1^{q(\varepsilon+1)} w_1\right)^{\frac{r}{p}} \left(V_1(x)^{-p'-\varepsilon r}-V_1(b)^{-p'-\varepsilon r} \right)^{\frac{q}{p}} V_1(x)^{qp'}w_1(x) dx \right)^{\frac{1}{q}}  
    \\
& +  \left( \int_a^b \left(\int_a^x V_1^{q(\varepsilon+1)} w\right)^{\frac{r}{p}} V_1(b)^{(-p'-\varepsilon r)\frac{q}{p}} V_1(x)^{qp'}w_1(x) dx \right)^{\frac{1}{q}}  \\
& \lesssim C \widetilde{B}_{2,\varepsilon}^{\frac{r}{p}} + \left(\int_a^b V_1^{q(\varepsilon+1)} w_1\right)^{\frac{r}{pq}} V_1(b)^{(-p'-\varepsilon r)\frac{1}{p}}\left( \int_a^b   V_1(x)^{qp'}w_1(x) dx \right)^{\frac{1}{q}} \\
& \lesssim C \widetilde{B}_{2,\varepsilon}^{\frac{r}{p}} + C \left(\int_a^b V_1^{q(\varepsilon+1)} w_1\right)^{\frac{r}{pq}} V_1(b)^{(-p'-\varepsilon r)\frac{1}{p}} V_1(b)^{\frac{p'}{p}} \\
&\lesssim  C \widetilde{B}_{2,\varepsilon}^{\frac{r}{p}}, 
\end{align*}
in which the multiplicative constants depend only on $p, \, q$ and $\varepsilon$. Since $\widetilde{B}_{1,\varepsilon}\approx \widetilde{B}_{2,\varepsilon}$ owing to integration by parts, we obtain 
\begin{equation*}
    \widetilde{B}_{1,\varepsilon}=\widetilde{B}_{1,\varepsilon}^{\frac rq-\frac rp} \lesssim C.
\end{equation*}
By approximating $v$ with an almost everywhere pointwise decreasing sequence and $w$ with an almost everywhere pointwise increasing one, and applying the Monotone Convergence Theorem, we arrive, owing to the fact that the mulciplicative constants of the above estimates do not depend on weights involved, at the estimate
\begin{equation*}
    {B}_{1,\varepsilon} \lesssim C,
\end{equation*}
which yields the desired result, namely that ${B}_{1,\varepsilon}<\infty$.

Finally, let $p=1$. Let $0 < w_1 \le w$ and $0 < v \le v_1$ be such that $\int_a^b w_1 <\infty$ and $\esssup\limits_{s\in (a,b)}\frac{1}{v_1(s)}<\infty$, set 
\begin{equation} V_1(t):=\esssup\limits_{s\in (a,t)}\frac{1}{v_1(s)},
\end{equation}
and let $\widetilde{B}_{1,\varepsilon}$ and $\widetilde{B}_{2,\varepsilon}$ be given by~\eqref{B-1-tilde} and~\eqref{B-2-tilde}, with $p=1$, respectively. By estimate  \eqref{est_B2}  we have $\widetilde{B}_{2,\varepsilon}<\infty$.  Then
\begin{equation}\label{E:hardy-new_p=1}
\left(\int_a^b \left(\int_a^{t}f\right)^{q} w_1(t) \, dt\right)^{\frac{1}{q}} \leq C\int_a^b fv_1
\end{equation}
holds for all $f \in \mathcal{M}^+(a, b)$. 
Fix $\sigma>1$ and define
\begin{equation*}
    E_k = \{t\in(a,b) : \sigma^k<V_1(t)\le \sigma^{k+1}\} \quad \text{for $k\in\Z$.}
\end{equation*}
Set $\mathbb A = \{k\in\Z : |E_k|>0\}$.  It clear that $A$ is bounded from above. Then  $(a,b)=\bigcup_{k\in\mathbb A}E_k$, in which, once again, the equality means up to a set of measure zero, in which the union is disjoint and each $E_k$ is a nondegenerate interval (which could be either open or closed at each end) with endpoints $a_k$ and $b_k$, $a_k<b_k$. For every $k\in\mathbb A$, we find $\delta_k > 0$ so that $a_k+\delta_k<b_{k}$ and
\begin{align}
\int_{a_k}^{b_k}&\left(\int_a^tV_1^{q(\varepsilon +1)}w\right)^{\frac{q}{1-q}}
        V_1(t)^{-\frac{\varepsilon q}{1-q}+q(\varepsilon +1)}w_1(t)\,dt\notag \\
        &\le \sigma \int_{a_k+\delta_k}^{b_k} \left(\int_a^tV_1^{q(\varepsilon +1)}w_1\right)^{\frac{q}{1-q}}
        V_1(t)^{-\frac{\varepsilon q}{1-q}+q(\varepsilon +1)}w_1(t)\,dt,
 \label{E:small-epsilons}   
\end{align}
which is clearly possible, and then we define the set
\begin{equation*}
    G_k = \left\{t\in(a_k, a_k+\delta_k) : \frac{1}{v_1(t)} >\sigma^{k}\right\}.
\end{equation*}
Since $V_1$ is non-decreasing and left-continuous, $|G_k|>0$ for every $k\in\mathbb A$. Set $h = \sum_{k\in\mathbb A} \frac{\chi_{_{G_k}}}{|G_k|}$. 
Define  
\[
    f(y)=h(y) V_1(y)^{\varepsilon +1}\left(\int_y^b\left(\int_a^xV_1^{q(\varepsilon +1)}w_1\right)^{\frac{1}{1-q}}
    \, d\left[-V_1(x)^{-\frac{\varepsilon }{1-q}} \right]\right), \quad y\in (a,b) .
\]
        
By using Fubini's theorem, we have  
\begin{align} 
    \int_{a}^{b} fv_1 
    &=
    \int_{a}^{b}  h(t) V_1(t)^{\varepsilon +1}\left(\int_t^b\left(\int_a^xV_1^{q(\varepsilon +1)}w_1\right)^{\frac{1}{1-q}}
            \, d\left[-V_1(x)^{-\frac{\varepsilon }{1-q}} \right]\right)v_1(t)\,dt\notag
            \\
                  & =
    \int_{a}^{b}  \left(\int_a^xV_1^{q(\varepsilon +1)}w_1\right)^{\frac{1}{1-q}}\int_a^x h(t) V_1(t)^{\varepsilon +1}v_1(t)\,dt
            \, d\left[-V_1(x)^{-\frac{\varepsilon }{1-q}} \right].\label{leftEst}
\end{align}

It is clear that if $x\in E_k$, then     
\begin{equation*}
    \int_{a}^{x}hV_1^{\varepsilon +1} v_1 \lesssim \sum_{i=-\infty}^k\frac{1}{|G_i|}\int_{G_i}V_1^{\varepsilon +1}v_1 \lesssim \sum_{i=-\infty}^k\sigma^{i \varepsilon}\approx  \sigma^{k \varepsilon} \approx V_1(x)^{\varepsilon}.
\end{equation*}
As $(a,b)=\bigcup_{k\in\mathbb A}E_k$, up to a set of measure zero, in which the union is disjoint,  
\begin{equation}\label{E:estimate-of-h}
    \int_{a}^{x}hV_1^{\varepsilon +1} v_1 \lesssim V_1(x)^{\varepsilon}, \quad \text{for every}\quad x\in (a,b).
    \end{equation}
If $x\in E_k$, we also have 
    \begin{equation} \label{E:revestimate-of-h}V_1(x)^{\varepsilon+1}\approx  \sigma^{k( \varepsilon+1)}\approx
    \sum_{i=-\infty}^{k-1}\sigma^{i (\varepsilon +1)} \approx\sum_{i=-\infty}^{k-1}\frac{1}{|G_i|}\int_{G_i}V_1^{\varepsilon +1}\lesssim
    \int_{a}^{x}hV_1^{\varepsilon +1}.
\end{equation}
Moreover,
\begin{equation}\label{E:upper-estimate-of-integral}
    \int_{t}^{b}V_1(s)^{\varepsilon}\, d\left[-V_1(s)^{-\frac{\varepsilon }{1-q}} \right]\approx\int_{t}^{b}\,d\left[-V_1(s)^{-\frac{\varepsilon q}{1-q}} \right]\quad\text{for}\quad  t\in(a,b).
\end{equation}
Since $\left(\int_a^tV_1^{q(\varepsilon +1)}w_1\right)^{\frac{1}{1-q}}$ is non-decreasing, we can apply the variant of Hardy's lemma noticed in~\cite{Hei:93} (whose version for Lebesgue integrals can be found in~\cite[Chapter~2, Proposition~3.6]{Ben:88} - note that the proof presented there works verbatim for Lebesgue--Stieltjes integrals) to~\eqref{E:upper-estimate-of-integral} and get
\begin{align}
    \int_{a}^{b}V_1(t)^{\varepsilon }& \left(\int_a^tV_1^{q(\varepsilon +1)}w_1\right)^{\frac{1}{1-q}}\,d\left[-V_1(t)^{-\frac{\varepsilon }{1-q}}\right]\nonumber
        \\
    &\qquad\lesssim \int_{a}^{b}\left(\int_a^tV_1^{q(\varepsilon +1)}w_1\right)^{\frac{1}{1-q}}\,d\left[-V_1(t)^{-\frac{\varepsilon q}{1-q}}\right].
    \label{E:hardy-lemma-estimate}
\end{align}
Using \eqref{leftEst}, \eqref{E:estimate-of-h} and \eqref{E:hardy-lemma-estimate} we obtain 
\begin{align*}
\int_{a}^{b} fv_1 &\lesssim \int_{a}^{b}   V_1(x)^{\varepsilon }\left(\int_a^xV_1^{q(\varepsilon +1)}w_1\right)^{\frac{1}{1-q}}
        \, d\left[-V_1(x)^{-\frac{\varepsilon }{1-q}} \right]\\
                     &\lesssim \int_{a}^{b}   \left(\int_a^xV_1^{q(\varepsilon +1)}w_1\right)^{\frac{1}{1-q}}
        \, d\left[-V_1(x)^{-\frac{\varepsilon q}{1-q}} \right]
            \\
        &\le\widetilde{B}_{2,\varepsilon}^{r}.
\end{align*}      
Therefore, we have 
     \begin{equation}\label{E:fw<}
         \int_{a}^{b} fv_1\lesssim \widetilde{B}_{2,\varepsilon}^{r}.
     \end{equation}       
Using the decomposition of $(a,b)$ into $\bigcup_{k\in\mathbb A}E_k$, the definition of $E_k$, the fact that each $E_k$ is an interval with endpoints $a_k$, $b_k$, and  \eqref{E:small-epsilons}, we get          
\begin{align*}
  \widetilde{B}_{1,\varepsilon}^r  
    & =
    \int_{a}^{b} V_1(t)^{-\frac{\varepsilon q}{1-q}} \left(\int_a^tV_1^{q(\varepsilon +1)}w_1\right)^{\frac{q}{1-q}}
        V_1(t)^{q(\varepsilon +1)}w_1(t)\,dt
            \\
    &=\sum_{k\in\mathbb A} 
    \int_{a_k}^{b_k} V_1(t)^{-\frac{\varepsilon q}{1-q}} \left(\int_a^tV_1^{q(\varepsilon +1)}w_1\right)^{\frac{q}{1-q}}
       V_1(t)^{q(\varepsilon +1)}w_1(t)\,dt
            \\
    &\lesssim \sum_{k\in\mathbb A} 
    \int_{a_k+\delta_k}^{b_k} V_1(t)^{-\frac{\varepsilon q}{1-q}} \left(\int_a^tV_1^{q(\varepsilon +1)}w_1\right)^{\frac{q}{1-q}}
        V_1(t)^{q(\varepsilon +1)}w_1(t)\,dt
        \\
     &\lesssim \sum_{k\in\mathbb A} 
    \int_{a_k+\delta_k}^{b_k}\left(\int_a^tV_1^{q(\varepsilon +1)}w_1\right)^{\frac{q}{1-q}}
        \left(V_1(t)^{-\frac{\varepsilon }{1-q}}-V_1(b)^{-\frac{\varepsilon }{1-q}} \right)^q V_1(t)^{q(\varepsilon +1)}w_1(t)\,dt
        \\
     & +\sum_{k\in\mathbb A} 
\int_{a_k+\delta_k}^{b_k}\left(\int_a^tV_1^{q(\varepsilon +1)}w_1\right)^{\frac{q}{1-q}}
        V_1(b)^{-\frac{\varepsilon q}{1-q}}V_1(t)^{q(\varepsilon +1)}w_1(t)\,dt.
\end{align*}
Thus, using \eqref{E:revestimate-of-h}, \eqref{E:estimate-of-h}, \eqref{E:hardy-new_p=1} and the monotonicity of functions given by integrals, we obtain  
\begin{align*}
    \widetilde{B}_{1,\varepsilon}^r  &
                  \lesssim \sum_{k\in\mathbb A} 
\int_{a_k+\delta_k}^{b_k} \left(\int_a^t h V_1^{\varepsilon +1} \right)^q \left(\int_a^tV_1^{q(\varepsilon +1)}w_1\right)^{\frac{q}{1-q}}\left(\int_t^b 
        \, d\left[-V_1(x)^{-\frac{\varepsilon }{1-q}} \right]\right)^q w_1(t)\,dt
        \\
          & +\sum_{k\in\mathbb A} 
\left(\int_a^bV_1^{q(\varepsilon +1)}w_1\right)^{\frac{q}{1-q}}
        V_1(b)^{-\frac{\varepsilon q}{1-q}} \int_{a_k}^{b_k}\left(\int_a^t hV_1^{\varepsilon +1}\right)^q w_1(t)\,dt
        \\
       &\lesssim 
       \sum_{k\in\mathbb A} 
\int_{a_k+\delta_k}^{b_k} \left(\int_a^t h(y) V_1(y)^{\varepsilon +1}\left(\int_y^b\left(\int_a^xV_1^{q(\varepsilon +1)}w_1\right)^{\frac{1}{1-q}}
        \, d\left[-V_1(x)^{-\frac{\varepsilon }{1-q}} \right]\right)\,dy\right)^qw_1(t)\,dt
         \\
         & +  \left(\int_a^bV_1^{q(\varepsilon +1)}w_1\right)^{\frac{q}{1-q}}
        V_1(b)^{-\frac{\varepsilon q}{1-q}}\int_{a}^b \left(\int_a^t hV_1^{\varepsilon +1}\right)^qw_1(t)\,dt  
        \\
         &\lesssim \sum_{k\in\mathbb A} 
\int_{a_k}^{b_k} \left(\int_a^t h(y) V_1(y)^{\varepsilon +1}\left(\int_y^b\left(\int_a^xV_1^{q(\varepsilon +1)}w_1\right)^{\frac{1}{1-q}}
        \, d\left[-V_1(x)^{-\frac{\varepsilon }{1-q}} \right]\right)\,dy\right)^qw_1(t)\,dt
        \\
          & +C^q\left(\int_a^bV_1^{q(\varepsilon +1)}w_1\right)^{\frac{q}{1-q}}
        V_1(b)^{-\frac{\varepsilon q}{1-q}}\left(\int_{a}^b hV_1^{\varepsilon+1}v_1\right)^{q}
        \\  
        &\lesssim
\int_{a}^{b} \left(\int_a^t h(y) V_1(y)^{\varepsilon +1}\left(\int_y^b\left(\int_a^xV_1^{q(\varepsilon +1)}w_1\right)^{\frac{1}{1-q}}
        \, d\left[-V_1(x)^{-\frac{\varepsilon }{1-q}} \right]\right)\,dy\right)^qw_1(t)\,dt
        \\
        & +C^q\left(\int_a^bV_1^{q(\varepsilon +1)}w_1\right)^{\frac{q}{1-q}}
        V_1(b)^{-\frac{\varepsilon q}{1-q}}\left(\int_{a}^b hV_1^{\varepsilon+1}v_1\right)^{q}
        \\
                 &= \int_{a}^{b} \left(\int_a^t f\right)^qw_1(t)\,dt
 + C^q\left(\int_a^bV_1^{q(\varepsilon +1)}w_1\right)^{\frac{q}{1-q}}
        V_1(b)^{-\frac{\varepsilon q^2}{1-q}}.
\end{align*}
Consequently, by \eqref{E:hardy-new_p=1} and \eqref{E:fw<}, we finally arrive at
\begin{align*}
\widetilde{B}_{1,\varepsilon}^r  
& \lesssim C^q\left(\int_{a}^{b}  fv_1\right)^{q}
 + C^q \left(\int_a^bV_1^{q(\varepsilon +1)}w_1\right)^{\frac{q}{1-q}}V_1(b)^{-\frac{\varepsilon q^2}{1-q}}\lesssim C^q\widetilde{B}_{2,\varepsilon}^{rq},
\end{align*}
in which the multiplicative constants depend only on $q$ and $\varepsilon$. 
Since $\widetilde{B}_{1,\varepsilon}\approx \widetilde{B}_{2,\varepsilon}$ owing to integration by parts, dividing both sides of the latter estimate by $\widetilde{B}_{2,\varepsilon}^{rq}$, using the fact that $r-rq=q$, and taking the $q$-th roots, we finally obtain $\widetilde{B}_{1,\varepsilon} \lesssim C$. Once again, approximating $v$ with an almost everywhere pointwise decreasing sequence and $w$ with an almost everywhere pointwise increasing one, using the fact that the constants in estimates do not depend on weights, and applying the Monotone Convergence Theorem, we establish that $B_{1,\varepsilon}\lesssim C$, which in turn yields $B_{i,\varepsilon}\lesssim C$, $i=1,2$. The proof is complete.

\end{proof}

\bibliographystyle{abbrv}

\vskip+1cm

\end{document}